\documentclass[12pt]{amsart}
\usepackage{graphics}
\usepackage{amssymb}
\theoremstyle{plain}
 \newtheorem{thm}{Theorem}[section]
 \newtheorem{cor}[thm]{Corollary}
 \newtheorem{lem}[thm]{Lemma}
 \newtheorem{prop}[thm]{Proposition}

\theoremstyle{definition}
 \newtheorem{ex}{Example}[section]

\theoremstyle{remark}

\begin{document}
\title[On the fixed point spaces of some completely positive maps]
{On the fixed point spaces of some completely positive maps}

\author[Tomohiro Hayashi]{{Tomohiro Hayashi} }
\address[Tomohiro Hayashi]
{Nagoya Institute of Technology, 
Gokiso-cho, Showa-ku, Nagoya, Aichi, 466-8555, Japan}

\email[Tomohiro Hayashi]{hayashi.tomohiro@nitech.ac.jp}

\baselineskip=17pt

\maketitle

\begin{abstract} 
In this paper we generalize the results shown by Das and Peterson. 
Let $M$ be a ${\rm II}_1$-factor acting on $L^2(M)$. We consider 
certain unital normal completely positive maps on $B(L^2(M))$ 
which are identity on $M$. 
We 
investigate their fixed point spaces 
and obtain a rigidity result. 
As an application, we show some results of subfactors. 
\end{abstract}

\subjclass{.}

\keywords{}

\section{Introduction}

Let $M$ be a ${\rm II}_1$-factor acting on $H=L^2(M)$. 
Let $\{z_n\}_{n=1}^\infty$ be a sequence in $M$ such that 
 $\{z_n\}_{n=1}^\infty$ is $*$-closed as a set and 
$\sum_{n=1}^\infty z_nz_n^*=\sum_{n=1}^\infty z_n^*z_n=1$. 
We define unital normal completely positive maps 
$\phi=\sum_{n=1}^\infty {\rm Ad} Jz_nJ$ 
and $\tilde{\phi}=\sum_{n=1}^\infty {\rm Ad} z_n^*$. 
In the paper \cite{DP}, 
Das and Peterson 
investigate the fixed point space $B(H)^\phi$ and they 
showed the following interesting results. If $M$ is generated by  $\{z_n\}_{n=1}^\infty$, 
then we have 
\begin{enumerate}
\item $M'\cap B(H)^\phi={\Bbb C}$. 
\item $B(H)^\phi\cap B(H)^{\tilde{\phi}}={\Bbb C}$.  
\item If $\Phi$ is a normal unital completely positive map from  $B(H)^\phi$ 
to  $B(H)^\phi$ such that $\Phi=id$ on $M$, then we have 
 $\Phi=id$ on $B(H)^\phi$.  
\item If $M\not =B(H)^\phi$, then $B(H)^\phi$ is an AFD type ${\rm III}$-factor. 
\end{enumerate}
The aim of this paper is to generalize these results without the assumption that 
$\{z_n\}_{n=1}^\infty$ generates $M$. Moreover we investigate 
the properties of the fixed point $\bigcap_{j=1}^N B(H)^{\phi_j}$. 
Consider the family of sequences 
\begin{align*}
{\frak S}
=\{\{z_n\}_{n=1}^\infty:&\ z_n\in M,\ 
\sum_{n=1}^\infty z_nz_n^*=\sum_{n=1}^\infty z_n^*z_n=1,\\ 
&{\text{for each $n$, there exists $k$ such that $z_n^*=z_k$}}
\}
\end{align*}
and set  
$
G=\{
\phi=\sum_{n=1}^\infty {\rm Ad} Jz_nJ:\ \{z_n\}_{n=1}^\infty\in{\frak S}
\}
$. Fix an element $\phi\in G$ with 
$\phi=\sum_{n=1}^\infty{\rm Ad}Jz_nJ$ and denote 
by $P$ the von Neumann algebra generated by $\{z_n\}_{n}$. 
Then we have the inclusions 
$$
P\subseteq M\subseteq \langle
M,e_P
\rangle\subseteq B(H)^{\phi}, 
$$ 
where $e_P$ is the Jones projection onto $L^2(P)$ \cite{J}. 
We will 
show the following. 
\begin{enumerate}
\item $M'\cap B(H)^{\phi}=M'\cap\langle M,e_P\rangle
=J(P'\cap M)J$ (Theorem 2.3). 
\item If 
$P$ is a factor, then 
$e_P(B(H)^{\phi}\cap B(H)^{\tilde{\phi}})e_P=
{\Bbb C}e_P$. (Lemma 3.7.)
\item 
If $\Phi$ is a normal unital completely positive map from  $B(H)^{\phi}$ 
to  $B(H)^{\phi}$ such that $\Phi=id$ on $\langle M,e_P\rangle$, then we have 
 $\Phi=id$ on $B(H)^\phi$ (Theorem 3.3.) 
\item If $P$ is a non-AFD factor, then $B(H)^{\phi}$ is an AFD type 
${\rm III}$-factor. (Corollary 3.9.)
\item
Fix finite elements $\phi_1,\phi_2,\cdots,\phi_N\in G$ with 
$\phi_i\circ\phi_j=\phi_j\circ\phi_i$
. Then there exists elements 
$z_n^{(j)}\in M$ such that $\phi_j=\sum_{n=1}^\infty{\rm Ad}Jz_n^{(j)}J$, 
$\sum_{n=1}^\infty z_n^{(j)}{z_n^{(j)}}^*=\sum_{n=1}^\infty {z_n^{(j)}}^*z_n^{(j)}=1$ and 
${z_n^{(j)}}^*=z_k^{(j)}$ for some $k$. For any positive numbers $\{\lambda_j\}_{j=1}^N$ with 
$\sum_{j=1}^N\lambda_j=1$, we set $\phi=\sum_{j=1}^N \lambda_j\phi_j$. Then we have 
$\bigcap_{j=1}^N B(H)^{\phi_j}=B(H)^\phi$. (Proposition 3.4.)

\end{enumerate}  
For the proof, we basically follow the argument in \cite{DP} with some modifications. 
As an application, we investigate the inclusions $B(H)^{\phi^k}\subseteq B(H)^{\phi^{mk}}$. 
We determine the Jones index and obtain some properties for these inclusions.

\section{basic properties}

Throughout this paper, let $M$ be a ${\rm II}_1$-factor with a unique tracial state $\tau$. 
We set $H=L^2(M,\tau)$ with a cyclic trace unit vector $\xi_0$.  
We consider the family of sequences 
\begin{align*}
{\frak S}
=\{\{z_n\}_{n=1}^\infty:&\ z_n\in M,\ 
\sum_{n=1}^\infty z_nz_n^*=\sum_{n=1}^\infty z_n^*z_n=1,\\ 
&{\text{for each $n$, there exists $k$ such that $z_n^*=z_k$}}
\}
\end{align*}
and set  
$
G=\{
\phi=\sum_{n=1}^\infty {\rm Ad} Jz_nJ:\ \{z_n\}_{n=1}^\infty\in{\frak S}
\}
$. 
Then each element $\phi$ in $G$ is a unital normal completely positive map 
from $B(H)$ to $B(H)$ such that $\phi=id$ on $M$. Here we
remark 
that $G$ is not closed under composition. 
For each $\phi=\sum_{n=1}^\infty{\rm Ad}Jz_nJ\in G$, we set 
$
\tilde{\phi}=\sum_{n=1}^\infty{\rm Ad}z_n^*
$. Then $\tilde{\phi}$ is also a unital normal completely positive map.

Let $P$ be a von Neumann algebra generated by $\{z_n\}_{n=1}^\infty$ and 
let $e_P$ be the Jones projection onto $L^2(P,\tau)$. Consider the Jones basic extension 
$$
P\subseteq M \subseteq\langle M,e_P\rangle=(JPJ)'.
$$
Let $B(H)^{\phi}$ be the set of fixed points of $\phi$. Since 
$Je_PJ=e_P$ and $e_P\in P'$, we see that $\phi(e_P)=
\sum_{n=1}^\infty Jz_nJe_PJz_n^*J=e_P\sum_{n=1}^\infty Jz_nz_n^*J=e_P
$. Thus we have the inclusions 
$$
P\subseteq M \subseteq\langle M,e_P\rangle\subseteq B(H)^{\phi}.
$$
Let $\omega_{0}$ be the vector state of a trace vector $\xi_0$. 
Then it is easy to see that 
$\omega_0\circ\phi=\omega_0\circ\tilde{\phi}$ on $B(H)$.

\begin{lem}
 $\omega_{0}$ is faithful on $e_PB(H)^{\phi}e_P$. 

\end{lem}

\begin{proof}
Let $T\in B(H)^{\phi}$ be a positive element satisfying 
$\omega_0(e_P Te_P)=0$. 
Since $e_P Te_P\in B(H)^{\phi}$, we see that 
$$
0=\omega_0(e_P Te_P)=\langle e_P Te_P\xi_0,\xi_0 \rangle
=\langle \phi^{n}(e_P Te_P)\xi_0,\xi_0 \rangle\\
=\sum ||T^{\frac{1}{2}}
w
\xi_0||^2,
$$
where $w$ is some word of $\{z_n\}_{n}$. Since the set $\{z_n\}_{n=1}^\infty$ is 
$*$-closed, the word $w$ runs through all $*$-words of  $\{z_n\}_{n}$. 
Thus we conclude that $T^{\frac{1}{2}}e_P=0$ and hence $e_P Te_P=0$. 
\end{proof}

\begin{lem}
If $T,T^*T\in B(H)^{\phi}$, then we have 
$T\in (\{Jz_nJ\}_{n=1}^\infty)'=\langle M,e_P\rangle$. 
In particular, $\langle M,e_P\rangle$ is a maximum $*$-subalgebra 
of $B(H)^{\phi}$. 

Similarly, 
if $T,T^*T\in B(H)^{\tilde{\phi}}$, then we have 
$T\in (\{z_n\}_{n=1}^\infty)'$.
\end{lem}

\begin{proof}
Assume that $T,T^*T\in B(H)^{\phi}$. 
We observe that for any $\xi\in H$, 
\begin{align*}
&||(Jz_n^*JT-TJz_n^*J)\xi||^2=
\langle
(T^*Jz_nJ-Jz_nJT^*)(Jz_n^*JT-TJz_n^*J)\xi.\xi
\rangle\\
&=\langle
(T^*Jz_nz_n^*JT+Jz_nJT^*TJ{z_n}^*J-T^*Jz_nJTJz_n^*J
-Jz_nJT^*Jz_n^*JT)\xi,
\xi\rangle
\end{align*}
Then we have
$$
\sum_{n=1}^\infty||(Jz_n^*JT-TJz_n^*J)\xi||^2=
\langle
(T^*T+\phi(T^*T)-T^*\phi(T)-\phi(T^*)T)\xi,\xi
\rangle=0.
$$
Since the set $\{z_n\}_{n=1}^\infty$ is 
$*$-closed, we conclude that $T\in (\{Jz_nJ\}_{n=1}^\infty)'=(JPJ)'=\langle M,e_P\rangle$. 
Next we will show that $\langle M,e_P\rangle$ is a maximum $*$-subalgebra 
of $B(H)^{\phi}$. 
Let $A$ be a subset of $B(H)^\phi$ such that $A$ is a $*$-subalgebra 
of $B(H)$. Then for any $T\in A$, since $T,T^*T\in A\subseteq B(H)^\phi$, 
we have $T\in \langle M,e_P\rangle$. Therefore we conclude 
$A\subseteq \langle M,e_P\rangle$.

By the similar way, if $T,T^*T\in B(H)^{\tilde{\phi}}$, then we have 
$$
\sum_{n=1}^\infty||(z_nT-Tz_n)\xi||^2=
\langle
(T^*T+\tilde{\phi}(T^*T)-T^*\tilde{\phi}(T)-\tilde{\phi}(T^*)T)\xi,\xi
\rangle=0
$$ 
and hence 
$T\in (\{z_n\}_{n=1}^\infty)'$.
\end{proof}

\begin{thm}
$M'\cap B(H)^{\phi}=M'\cap\langle M,e_P\rangle
=J(P'\cap M)J$. 
\end{thm}

\begin{proof}
For any $x\in M$, we see that 
\begin{align*}
\tau(\phi(JxJ))&=\sum_{n=1}^\infty\langle
Jz_nxz_n^*J\xi_0,\xi_0
\rangle=\sum_{n=1}^\infty\tau(z_nx^*z_n^*)\\
&=\sum_{n=1}^\infty\tau(x^*z_n^*z_n)=\tau(x^*)
=\tau(JxJ).
\end{align*}
Moreover we have $\phi(M')\subseteq M'$. Thus $\phi$ is a 
unital normal completely positive map from $M'$ to $M'$ which preserves 
the faithful traicial state. Then by lemma 2.5 in 
\cite{DP} the fixed point space 
$(M')^{\phi}=M'\cap B(H)^{\phi}$ is a $*$-algebra. 
Therefore $M'\cap B(H)^{\phi}$ is a $*$-subalgebra of 
$B(H)^{\phi}$. 
Then by the previous lemma we have $M'\cap B(H)^{\phi}\subseteq 
\langle
M,e_P
\rangle
$. 
\end{proof}

\section{rigidity}

For any $T\in B(H)$, we set 
$$
E(T)=\lim_{n\rightarrow\omega}\dfrac{1}{n}\sum_{j=1}^n\phi^j(T),\ \ \ \ 
\tilde{E}(T)=\lim_{n\rightarrow\omega}\dfrac{1}{n}\sum_{j=1}^n{\tilde{\phi}}^j(T),
$$ 
where the limits are taken 
with respect to the weak operator topology and $\omega$ is some free ultrafilter. 
Then $E$ is a conditional expectation from $B(H)$ onto $B(H)^{\phi}$ and   
$\tilde{E}$ is a conditional expectation from $B(H)$ onto $B(H)^{\tilde{\phi}}$. 
Then $B(H)^{\phi}$ is 
an AFD von Neumann algebra via the Choi-Effros product 
$
S\circ T=E(ST)
$. 
Here we remark that $S\circ z=Sz$ and $z\circ S=zS$ 
for any $z\in \langle M,e_P\rangle$ because 
$\langle M,e_P\rangle$ is in the multiplicative domain of $\phi$ and $\phi=id$ on 
$\langle M,e_P\rangle$. In particular 
$x\circ y=xy$ for any $x,y\in \langle M,e_P\rangle$. 
Similarly,  $B(H)^{\tilde{\phi}}$ is 
an AFD von Neumann algebra via the Choi-Effros product 
$
S\circ T=\tilde{E}(ST)
$. (See \cite{I}.)

\begin{prop}
If $P$ is a factor, then $B(H)^{\phi}$ is also a factor via the 
Choi-Effros product. If $P\subseteq M$ is irreducible, 
the inclusion $M\subseteq B(H)^{\phi}$ is also irreducible. 
\end{prop}

\begin{proof}
By theorem 2.3, 
we see that $Z(B(H)^{\phi})
\subseteq M'\cap B(H)^{\phi}=M'\cap\langle M,e_P\rangle
$. Then we have 
$
Z(B(H)^{\phi})\subseteq 
(B(H)^{\phi})'\cap \langle M,e_P\rangle
\subseteq Z(\langle M,e_P\rangle)={\Bbb C}
$. 

Similarly, 
if $P'\cap M={\Bbb C}$, then we see that 
$
M'\cap B(H)^{\phi}=M'\cap\langle M,e_P\rangle
=J(P'\cap M)J={\Bbb C}
$. 
\end{proof}

\begin{lem}
$e_P(B(H)^\phi\cap B(H)^{\tilde{\phi}})e_P\subseteq\langle M,e_P\rangle
\cap \langle M',e_P\rangle$.

\end{lem}

\begin{proof}
Since $\tilde{\phi}(e_P)=e_P$ and $\phi\circ\tilde{\phi}=\tilde{\phi}\circ\phi$, 
$\tilde{\phi}$ preserves $e_PB(H)^{\phi}e_P$. 
For any $T\in e_PB(H)^{\phi}e_P$, we see that 
$$
\omega_0(\tilde{\phi}(T))=\omega_0(\phi(T))=\omega_0(T).
$$
Then for any $S\in (e_PB(H)^{\phi}e_P)^{\tilde{\phi}}$, 
since $S^*\circ S=\tilde{\phi}(S^*)\circ\tilde{\phi}(S)\leq \tilde{\phi}(S^*\circ S)$, we have 
$
\omega_0(S^*\circ S)\leq \omega_0(\tilde{\phi}(S^*\circ S))
=\omega_0(S^*\circ S)
$. 
Here recall that $S^*\circ S=E(S^*S)$. 
Since by lemma 2.1 $\omega_0$ is faithful on $
(e_PB(H)^{\phi}e_P)^{\tilde{\phi}}\subseteq 
e_PB(H)^{\phi}e_P$, we get 
$
 \tilde{\phi}(S^*\circ S)=S^*\circ S
$. Thus $(e_PB(H)^{\phi_j}e_P)^{\tilde{\phi}}$ is a von Neumann 
subalgebra of $B(H)^{\phi}$ via the Choi-Effross product. 

Take any projection $p\in (e_PB(H)^{\phi}e_P)^{\tilde{\phi}}$. That is, 
$p\geq0$ and $p\circ p=p$. 
We consider that $B(H)^{\phi}$ 
is a von Neumann subalgebra of $B(K)$. That is, the Choi-Effross producti is compatible 
with the usual product in $B(K)$. 
Since $\tilde{\phi}(p)=p$, 
for any $\xi\in K$ 
we see that 
\begin{align*}
\sum_{n=1}^\infty||p\circ z_n\circ(1-p)\xi||^2&=
\langle
\tilde{\phi}(p)\circ(1-p)\xi,(1-p)\xi)
\rangle\\ 
&=
\langle
p\circ(1-p)\xi,(1-p)\xi)
\rangle
=0
\end{align*}
and hence $p\circ  z_n= z_n\circ p$. 
Since $p\circ  z_n=pz_n$ and 
$z_n\circ p=z_np$, we conclude $p\in \{z_n\}_n'$. 
Therefore we have $(e_PB(H)^{\phi}e_P)^{\tilde{\phi}}\subseteq\{z_n\}_n'$. 
Since $(e_PB(H)^{\phi}e_P)^{\tilde{\phi}}=e_P(B(H)^\phi\cap B(H)^{\tilde{\phi}})e_P$ 
and $\{z_n\}_n'=P'=\langle M',e_P\rangle$, we conclude that 
$
e_P(B(H)^\phi\cap B(H)^{\tilde{\phi}})e_P\subseteq \langle M',e_P\rangle
$. 

Similarly we can show that 
$(e_PB(H)^{\tilde{\phi}}e_P)^{\phi}\subseteq\{J z_nJ\}'$. Since 
$(e_PB(H)^{\tilde{\phi}}e_P)^{{\phi}}=e_P(B(H)^\phi\cap B(H)^{\tilde{\phi}})e_P$ 
and $\{Jz_nJ\}_n'=(JPJ)'=\langle M,e_P\rangle$, we conclude that 
$
e_P(B(H)^\phi\cap B(H)^{\tilde{\phi}})e_P\subseteq \langle M,e_P\rangle
$.

\end{proof}

\begin{thm}
Let $\langle M,e_P\rangle\subseteq X \subseteq B(H)^{\phi}$ be a 
$\langle M,e_P\rangle$-$\langle M,e_P\rangle$ bimodule  
and let $\Phi$ be a unital normal linear map from $X$ to $B(H)^{\phi}$ such that 
$
\Phi(xTy)=x\Phi(T)y
$ for any $x,y\in \langle M,e_P\rangle$ and $T\in X$. 
Then $\Phi$ is identity on $X$. 

Similarly, let $Pe_P=e_P\langle M,e_P\rangle e_P\subseteq X \subseteq 
e_PB(H)^{\phi}e_P$ be a 
$Pe_P$-$Pe_P$ bimodule  
and let $\Phi$ be a unital normal 
linear map from $X$ to $e_PB(H)^{\phi}e_P$ such that 
$
\Phi(xTy)=x\Phi(T)y
$ for any $x,y\in Pe_P$ and $T\in X$. 
Then $\Phi$ is identity on $X$. 
\end{thm}

\begin{proof} 
We remark that by the assumptions $\Phi=id$ on $\langle M,e_P\rangle$ and 
$\tilde{\phi}\circ\Phi=\Phi\circ\tilde{\phi}$. 
For any $T\in X$, we observe that 
\begin{align*}
\langle
\Phi(T)\xi_0,\xi_0
\rangle&=
\dfrac{1}{n}\sum_{j=1}^n
\langle
\phi^j(\Phi(T))\xi_0,\xi_0
\rangle=\dfrac{1}{n}\sum_{j=1}^n
\langle
\tilde{\phi}^j(\Phi(T))\xi_0,\xi_0
\rangle\\
&=\dfrac{1}{n}\sum_{j=1}^n
\langle
\Phi(\tilde{\phi}^j(T))\xi_0,\xi_0
\rangle=
\langle
\Phi(\tilde{E}(T))\xi_0,\xi_0
\rangle\\
&=\langle
\Phi(\tilde{E}(T))e_P\xi_0,e_P\xi_0
\rangle=\langle
\Phi(e_P\tilde{E}(T)e_P)\xi_0,\xi_0
\rangle
\end{align*}
Since $T\in B(H)^{\phi}$and $\phi\circ\tilde{\phi}=\tilde{\phi}\circ\phi$, we have $\tilde{E}(T)\in 
(B(H)^{\phi})^{\tilde{\phi}}=B(H)^\phi\cap B(H)^{\tilde{\phi}}
$. Then by the previous lemma we see that 
$e_P\tilde{E}(T)e_P\in e_P(B(H)^\phi\cap B(H)^{\tilde{\phi}})e_P\subseteq 
\langle M,e_P\rangle$. Since $\Phi=id$ on $\langle M,e_P\rangle$, we have 
$$
\langle
\Phi(T)\xi_0,\xi_0
\rangle=\langle
e_P\tilde{E}(T)e_P\xi_0,\xi_0
\rangle.
$$
Since the identity map satisfies the same assumption as $\Phi$, we also have 
$$
\langle
T\xi_0,\xi_0
\rangle=\langle
e_P\tilde{E}(T)e_P\xi_0,\xi_0
\rangle
$$ and hence 
$$
\langle
\Phi(T)\xi_0,\xi_0
\rangle=\langle
T\xi_0,\xi_0
\rangle.
$$ 
Since $X$ is a $\langle M,e_P\rangle$-$\langle M,e_P\rangle$ bimodule, 
for any $x,y\in M$ we have $y^*Tx\in X$ and hence 
$$
\langle
\Phi(T)x\xi_0,y\xi_0
\rangle=
\langle
\Phi(y^*Tx)\xi_0,\xi_0
\rangle=\langle
y^*Tx\xi_0,\xi_0
\rangle
=\langle
Tx\xi_0,y\xi_0
\rangle.
$$
Therefore we conclude $\Phi(T)=T$. 

The rest of the statement can be shown by the same way. 
\end{proof}

\begin{prop}
Fix finite elements $\phi_1,\phi_2,\cdots,\phi_N\in G$ such that $\phi_i\circ\phi_j=\phi_j\circ\phi_i$. 
For any positive numbers $\{\lambda_j\}_{j=1}^N$ with 
$\sum_{j=1}^N\lambda_j=1$, we set $\phi=\sum_{j=1}^N \lambda_j\phi_j$. Then we have 
$\bigcap_{j=1}^N B(H)^{\phi_j}=B(H)^\phi$.
\end{prop}

\begin{proof} 
There exists elements 
$z_n^{(j)}\in M$ such that $\phi_j=\sum_{n=1}^\infty{\rm Ad}Jz_n^{(j)}J$, 
$\sum_{n=1}^\infty z_n^{(j)}{z_n^{(j)}}^*=\sum_{n=1}^\infty {z_n^{(j)}}^*z_n^{(j)}=1$ and 
${z_n^{(j)}}^*=z_k^{(j)}$ for some $k$. 
Since $\phi=\sum_{n,j}{\rm Ad}J(\sqrt{\lambda_j}z_n^{(j)})J$, we have $\phi\in G$. 
Let $P$ be a von Neumann algebra generated by $\{\sqrt{\lambda_j}z_n^{(j)})\}_{n,j}$. Then 
$P$ is generated by $\{z_n^{(j)}\}_{n,j}$. By theorem 3.3, the inclusion 
$\langle M,e_P\rangle\subseteq B(H)^\phi$ has the rigidity. Since 
$\langle M,e_P\rangle=(JPJ)'=\{Jz_n^{(j)}J\}_{n,j}'\subseteq \bigcap_{j=1}^N B(H)^{\phi_j}$, 
each map $\phi_j$ is identity on $\langle M,e_P\rangle$. Then 
since $\phi_j(B(H)^\phi)\subseteq B(H)^\phi$, 
by the rigidity  $\phi_j$ is 
also identity on $B(H)^\phi$. Therefore we conclude that $\bigcap_{j=1}^N B(H)^{\phi_j}=B(H)^\phi$. 

\end{proof}

\begin{prop}
Let $\phi=\sum_{n=1}^\infty{\rm Ad} Jz_nJ$ and $\psi=\sum_{n=1}^\infty {\rm Ad}Jw_nJ$ be 
elements in $G$ such that $z_nw_m=w_mz_n$. We furhter assume that 
$\sum_j z_{n_j}=1$ for some $n_1<n_2<\cdots$.  
Then we have 
$B(H)^\phi\cap B(H)^\psi=B(H)^{\phi\psi}$. 
\end{prop}

\begin{proof}
Let $Q$ be a von Neumann algebra generated by $\{z_nw_m\}_{n,m=1}^\infty$ and 
let $P$ be a von Neumann algebra generated by $\{z_n\}_{n=1}^\infty\cup\{w_m\}_{m=1}^\infty$. 
Then we have the following inclusions. 
$$
\begin{matrix}
Q&\subseteq& P&\subseteq &M&
\subseteq
&\langle 
M,e_P
\rangle&
\subseteq
&\langle 
M,e_Q
\rangle\\
 & & & & & &\cap& &\cap\\
 & & & & & & 
B(H)^{\phi}\cap B(H)^\psi&\subseteq &B(H)^{\phi\psi}
\end{matrix}. 
$$ 
We claim that $Q=P$. Since$\sum_{j}z_{n_j}=1$, 
we see that $w_m=\sum_{j}z_{n_j}w_m\in Q$. Then we have 
$
Q\ni\sum_{m=1}^\infty(z_nw_m)w_m^*=z_n
$ and hence $Q=P$. Then we also have 
$
\langle 
M,e_P
\rangle
=\langle 
M,e_Q
\rangle
$. Since $\phi\psi\in G$, the inclusion 
$\langle 
M,e_P
\rangle
=\langle 
M,e_Q
\rangle\subset B(H)^{\phi\psi}$ 
has the rigidity in the sense of theorem 3.3. Since 
$\phi(B(H)^{\phi\psi})\subseteq B(H)^{\phi\psi}$ 
and $\phi=id$ on $\langle 
M,e_P
\rangle$, by the rigidity we get $\phi=id$ on $B(H)^{\phi\psi}$. 
By the same reason we also obtain that  $\psi=id$ on $B(H)^{\phi\psi}$ 
and hence $B(H)^\phi\cap B(H)^\psi=B(H)^{\phi\psi}$.

\end{proof}

In general, $B(H)^\phi\cap B(H)^\psi\not=B(H)^{\phi\psi}$ 
even if $\phi,\psi\in G$ and $\phi\circ\psi=\psi\circ\phi$. 
See corollary 4.12.

\begin{cor}
Let $\phi=\sum_{n=1}^\infty {\rm Ad}Jz_nJ$ and $\psi=\sum_{n=1}^\infty {\rm Ad}Jw_nJ$ be 
elements in $G$ such that $z_nw_m=w_mz_n$. 
Then we have 
$B(H)^{\phi^2}\cap B(H)^\psi=B(H)^{\phi^2\psi}$. 
\end{cor}

\begin{proof} 
By the definition of $G$, for each $n$ there exists $k(n)$ such that 
$z_n^*=z_{k(n)}$. 
Since $\phi^2=\sum_{m,n=1}^\infty{\rm Ad} Jz_mz_nJ$ and $\sum_{n=1}^\infty 
z_nz_{k(n)}=\sum_{n=1}^\infty 
z_nz_n^*=1
$, the pair $\phi^2$ and $\psi$ satisfy the assumptions in the previous proposition. 
Therefore we have $B(H)^{\phi^2}\cap B(H)^\psi=B(H)^{\phi^2\psi}$. 
\end{proof}

\begin{lem}
If $P$ is a factor, then 
$e_P(B(H)^\phi\cap B(H)^{\tilde{\phi}})e_P=
{\Bbb C}e_P$. 
\end{lem}

\begin{proof}
By lemma 3.2, we have 
$e_P(B(H)^\phi\cap B(H)^{\tilde{\phi}})e_P\subseteq 
e_P\langle M,e_P\rangle
\cap \langle M',e_P\rangle e_P=Pe_P\cap JPJe_P=
Z(P)e_P={\Bbb C}e_P
$.
\end{proof}

Recall that if $P$ is a factor, 
then by Proposition 3.1 $B(H)^{\phi}$ is also a factor. 

\begin{thm}
Assume that $P$ is a factor. If $B(H)^{\phi}$ is a 
finite factor, then $B(H)^{\phi}=\langle 
M,e_P
\rangle$. 
If $B(H)^{\phi}$ is a 
${\rm II}_\infty$-factor, 
then $e_PB(H)^{\phi}e_P=Pe_P$.  

\end{thm}

\begin{proof}
If $B(H)^{\phi}$ is a 
finite factor, then there exists a trace-preserving conditional expectation 
$F$ from 
$B(H)^{\phi}$ onto $\langle M,e_P\rangle$. 
Then by theorem 3.3, we must have $F=id$ and hence 
$B(H)^{\phi}=\langle 
M,e_P
\rangle$. 

If $B(H)^{\phi}$ is a 
${\rm II}_\infty$-factor, then there exists a faithful normal semifinite tracial 
weight $Tr$. For any positive element $T\in B(H)^{\phi}$ 
with 
$0<Tr(T)<\infty$, we see that 
$$
Tr(T)=\dfrac{1}{n}\sum_{k=1}^n
Tr(\tilde{\phi}^k(T)).
$$ 
Take an increasing sequence of finite projections $\{P_j\}_{j=1}^\infty
\subseteq B(H)^{\phi}
$ with $\lim_{j\rightarrow\infty}P_j=1$. Then we have 
$$
\infty>Tr(T)=\dfrac{1}{n}\sum_{k=1}^n
Tr(\tilde{\phi}^k(T))\geq 
\dfrac{1}{n}\sum_{k=1}^n
Tr(\tilde{\phi}^k(T)P_j).
$$
First tending $n\rightarrow\omega$, after that tending $j\rightarrow \infty$, we have 
$$
\infty>Tr(T)\geq Tr(\tilde{E}(T))\geq Tr(e_P\tilde{E}(T)e_P).
$$
Since $T\in B(H)^{\phi}$, by the proof of theorem 3.3 
we have 
$$
\langle
T\xi_0,\xi_0
\rangle=\langle
e_P\tilde{E}(T)e_P\xi_0,\xi_0
\rangle.
$$ 
Since $Tr$ is semifinite, we can take $T$ with $e_PTe_P\not=0$. Then 
since $\omega_0$ is faithful on $e_PB(H)^{\phi}e_P$, we get 
$$
0\not=\omega_0(e_PTe_P)=\langle
T\xi_0,\xi_0
\rangle=\langle
e_P\tilde{E}(T)e_P\xi_0,\xi_0
\rangle.
$$
In particular we obtain $e_P\tilde{E}(T)e_P\not=0$. 
Since by lemma 3.7 
$$e_P\tilde{E}(T)e_P\in 
e_P(B(H)^\phi\cap B(H)^{\tilde{\phi}})e_P=
{\Bbb C}e_P,
$$ we have $e_P\tilde{E}(T)e_P=\lambda e_P$ for some number $\lambda>0$. 
Then we see that 
$$
\infty>Tr(T)\geq Tr(\tilde{E}(T))\geq Tr(e_P\tilde{E}(T)e_P)=
\lambda Tr(e_P).
$$ 
This means that $e_P$ is a finite projection in $B(H)^{\phi}$ and hence 
$e_PB(H)^{\phi}e_P$ is a finite factor. Then there exists a 
trace-preserving conditional expectation from $e_PB(H)^{\phi}e_P$ 
onto $e_P\langle M,e_P\rangle e_P=Pe_P$. Then by theorem 3.3, we conclude 
$e_PB(H)^{\phi}e_P=Pe_P$.  
\end{proof}

\begin{cor}
If $P$ is a non-AFD factor, then $B(H)^{\phi}$ is an AFD type 
${\rm III}$-factor.

\end{cor} 

\begin{proof}
Since $P$ is non-AFD, $\langle 
M,e_P
\rangle$ is also non-AFD. Then since $B(H)^{\phi}$ is 
AFD, we have $B(H)^{\phi}\not=\langle 
M,e_P
\rangle$. Then by the previous theorem,  $B(H)^{\phi}$ is a 
properly infinite factor. If  $B(H)^{\phi}$ is a ${\rm II}_\infty$-factor, 
then $e_PB(H)^{\phi}e_P=Pe_P$.  However since $Pe_P$ is non-AFD 
and $e_PB(H)^{\phi}e_P$ is AFD, 
this cannot occur and hence  $B(H)^{\phi}$ is a type 
${\rm III}$-factor.

\end{proof}

\section{Some applications to subfactors}
Let 
$\phi=\sum_{n=1}^\infty{\rm Ad}Jz_nJ\in G$ $(z_n\in M)$ such that  
$\sum_{n=1}^\infty z_nz_n^*=\sum_{n=1}^\infty z_n^*z_n=1$ 
and for each $n$, there exists $k$ such that $z_n^*=z_k$. 
Let $P_k$ be a von Neumann algebra generated by $\{z_{n_1}z_{n_2}\cdots z_{n_k}\}_{n_1,n_2,\cdots,n_k=1}^\infty$. 
Let $e_k$ be the Jones projection onto $L^2(P_k,\tau)$. 
We remark that 
$
\phi^k=\sum_{n_1,n_2,\cdots,n_k=1}^\infty{\rm Ad}Jz_{n_1}z_{n_2}\cdots z_{n_k}J\in G 
$. Then we have the following inclusions. 
$$
\begin{matrix}
P_{mk}&\subseteq& P_k&\subseteq &M&
\subseteq
&\langle 
M,e_k
\rangle&
\subseteq
&\langle 
M,e_{mk}
\rangle\\
 & & & & & &\cap& &\cap\\
 & & & & & & 
B(H)^{\phi^k}&\subseteq &B(H)^{\phi^{mk}}
\end{matrix}
$$ 
We further assume that ${P_2}'\cap P_1={\Bbb C}$ and 
$P_2\not=P_1$.

\begin{lem}
The inclusion $B(H)^{\phi^k}\subseteq B(H)^{\phi^{mk}}$ is isomorphic to 
an inclusion of von Neumann algebras as operator systems. That is, both $B(H)^{\phi^k}$ and $B(H)^{\phi^{mk}}$ 
have the same product as von Neumann algebras.
\end{lem}

\begin{proof}
This is a consequence of the Bhat dilation \cite{B96,B99} and the Prunaru theorem \cite{Pru12}. 
By the Bhat dilation, we have a Hilbert space $K\supset H$ and a unital 
normal $*$-endmorphism $\alpha$ on $B(K)$ such that 
$e_H\alpha^n(T)e_H=\phi^n(T)$ for any $T\in B(H)$, where 
$e_H$ is a projection onto $H$. Then by the Prunaru theorem, 
$e_H B(K)^{\alpha^k}e_H=B(H)^{\phi^k}$ and 
$e_H B(K)^{\alpha^k}e_H\subseteq e_H B(K)^{\alpha^{mk}}e_H
\simeq B(K)^{\alpha^k}\subseteq B(K)^{\alpha^{mk}}
$ as operator systems. 
 
\end{proof}

\begin{lem}
The inclusion $B(H)^{\phi}\subseteq B(H)^{\phi^2}$ is irreducible.
\end{lem}

\begin{proof}
By theorem 2.3, we have $M'\cap B(H)^{\phi}=M'\cap \langle 
M,e_1
\rangle$ and $M'\cap B(H)^{\phi^2}=M'\cap \langle 
M,e_2
\rangle$. Then we have 
$
{B(H)^{\phi}}'\cap B(H)^{\phi^2}={B(H)^{\phi}}'\cap M'\cap B(H)^{\phi^2}
={B(H)^{\phi}}'\cap M'\cap \langle M,e_2\rangle\subseteq \langle 
M,e_1\rangle'\cap \langle M,e_2\rangle=JP_2'\cap P_1J={\Bbb C}
$. 
\end{proof}

\begin{lem}
We set $F(T)=\frac{1}{2}(T+\phi(T))$. Then $F$ is a normal conditional 
expectation from $B(H)^{\phi^2}$ onto $B(H)^{\phi}$. Moreover 
$F(\langle 
M,e_2
\rangle)\subseteq\langle 
M,e_1
\rangle$.  
\end{lem}

\begin{proof}
The first statement is obvious. We will show that $F(\langle 
M,e_2
\rangle)\subseteq\langle 
M,e_1
\rangle$. For any $T\in \langle 
M,e_2
\rangle=(\{Jz_{n_1}z_{n_2}J\}_{n_1,n_2=1}^\infty)'$, we see that 
$$
Jz_kJ\phi(T)=\sum_{n=1}^\infty Jz_kz_nJTJz_n^*J=
\sum_{n=1}^\infty T (Jz_kz_nJ)Jz_n^*J=TJz_kJ\sum_{n=1}^\infty Jz_nz_n^*J=TJz_kJ.
$$
Since both  $\langle 
M,e_2
\rangle$ and $\{z_n\}_{n=1}^\infty$ are closed under the $*$-operation, 
we also have 
$
Jz_kJT=\phi(T)Jz_kJ
$. Then we have $Jz_kJF(T)=F(T)Jz_kJ$ and hence $F(T)\in (\{Jz_{n}J\}_{n=1}^\infty)'
=\langle 
M,e_1
\rangle$
\end{proof}

Recall that an irreducible subfactor has a unique normal conditional expectation if it exists. 
Since the inclusions $\langle 
M,e_1
\rangle
\subseteq
\langle 
M,e_{2}
\rangle$ and $B(H)^{\phi}\subseteq B(H)^{\phi^2}$ are irreducible, 
they have the unique conditional expectation $F$.

\begin{prop}
$[P_2:P_1]=[B(H)^{\phi^2}:B(H)^{\phi}]=2$.
\end{prop}

\begin{proof}
For any positive element $T\in B(H)^{\phi^2}$, we have 
$
F(T)=\frac{1}{2}(T+\phi(T))\geq \frac{1}{2}T
$. Therefore 
by the Pimsner-Popa theorem \cite{PP,Po}, 
both $[P_2:P_1]$ and $[B(H)^{\phi^2}:B(H)^{\phi}]$ are equal to 
either $1$ or $2$. Since $P_1\not=P_2$, we have $[P_2:P_1]=2$. 
We assume that $B(H)^{\phi^2}=B(H)^{\phi}$. By lemma 2.2, $\langle M,e_1\rangle$ is a 
maximum $*$-subalgebra of $B(H)^{\phi}$ and $\langle M,e_2\rangle$ is a 
maximum $*$-subalgebra of $B(H)^{\phi^2}$. Therefore we get 
$\langle M,e_1\rangle=\langle M,e_2\rangle$ and hence $P_1=P_2$. 
Since $P_1\not=P_2$, we conclude that  $B(H)^{\phi^2}\not=B(H)^{\phi}$ and 
$[B(H)^{\phi^2}:B(H)^{\phi}]=2$.
\end{proof}

\begin{prop}
Let $\{m_i\}_{i=1}^n\subset \langle 
M,e_2
\rangle$ be a Pimsner-Popa basis of the inclusion 
$\langle 
M,e_1
\rangle\subset \langle 
M,e_2
\rangle$. That is, for any $x\in \langle 
M,e_2
\rangle$, we have 
$
x=\sum_{i=1}^n m_iF(m_i^*x).
$ 
Then for any $T\in B(H)^{\phi^2}$, we also have 
$
T=\sum_{i=1}^n m_iF(m_i^*T).
$ 
\end{prop}

\begin{proof}
Let 
$$
\Phi(T)=\frac{1}{2}\sum_{i.j=1}^nm_iF(m_i^*Tm_j)m_j^*.
$$
Then $\Phi$ is a normal unital completely positive map on $B(H)^{\phi^2}$. 
Moreover for any  $x\in \langle 
M,e_2
\rangle$, we have 
$$
\Phi(x)=\frac{1}{2}\sum_{j=1}^n\{\sum_{i=1}^nm_iF(m_i^*xm_j)\}m_j^*
=\frac{1}{2}\sum_{j=1}^n xm_jm_j^*=x.
$$ 
Here we used the fact that $\sum_{i=1}^n m_im_i^*=[\langle 
M,e_2
\rangle:\langle 
M,e_1
\rangle]=2$. Then by theorem 3.3, we have $\Phi=id$ on  $B(H)^{\phi^2}$. 
For any positive element $T\in B(H)^{\phi^2}$, 
since $ (F(m_i^*Tm_j))_{i,j}\geq  \frac{1}{2}(m_i^*Tm_j)_{i,j}$, 
we observe that 
\begin{align*}
T&=\Phi(T)=\frac{1}{2}(m_1,\cdots,m_n)\times (F(m_i^*Tm_j))_{i,j}\times (m_1,\cdots,m_n)^*\\
&\geq\frac{1}{2}(m_1,\cdots,m_n)\times \frac{1}{2}(m_i^*Tm_j)_{i,j}\times (m_1,\cdots,m_n)^*
=T.
\end{align*}
Then we obtain 
$$
\{ (F(m_i^*Tm_j))_{i,j}-\frac{1}{2}(m_i^*Tm_j)_{i,j}\}\times(m_1,\cdots,m_n)^*=(0,\cdots,0)^*
$$ and hence 
$$
\sum_{j=1}^nF(m_i^*Tm_j)m_j^*=m_i^*T.
$$
Then for any element $S\in B(H)^{\phi^2}$, 
$$
\sum_{j=1}^nF(m_i^*Sm_j)m_j^*=m_i^*S.
$$ By letting $S=F(m_i)T$, we have 
$$
\sum_{j=1}^nF(m_i^*F(m_i)Tm_j)m_j^*=m_i^*F(m_i)T
$$ 
and hence 
$$
\sum_{j=1}^nF(\{\sum_{i=1}^nm_i^*F(m_i)\}Tm_j)m_j^*=\{\sum_{i=1}^nm_i^*F(m_i)\}T.
$$
Let $e$ be the Jones projection of the 
subfactor $\langle 
M,e_1
\rangle\subset \langle 
M,e_2
\rangle$. Then 
there exists a Pimsner-Popa basis $\{n_i\}_i\subset \langle 
M,e_2
\rangle$ of the inclusion 
$\langle 
M,e_1
\rangle\subset \langle 
M,e_2
\rangle$ such that 
$n_1=1$ and $(1-e)n_ie=0$ for $i\geq2$. 
Since $F(n_i)e=en_ie$, we have $F(n_i)=0$ for $i\geq2$. 
By the above argument we have 
$$
\sum_{j=1}^mF(\{\sum_{i=1}^nn_i^*F(n_i)\}Tn_j)n_j^*=\{\sum_{i=1}^mn_i^*F(n_i)\}T.
$$ 
Then since $\sum_{i=1}^nn_i^*F(n_i)=1$, we obtain 
$
\sum_{j=1}^mF(Tn_j)n_j^*=T.
$ 
Since $m_i=\sum_{j=1}^mn_jF(n_j^*m_i)$, we complute 
\begin{align*}
\sum_{i=1}^nF(Tm_i)m_i^*&=
\sum_{i=1}^n\sum_{j,k=1}^mF(Tn_j)F(n_j^*m_i)F(m_i^*n_k)n_k^*\\
&=\sum_{j,k=1}^mF(Tn_j)F(n_j^*\sum_{i=1}^nm_iF(m_i^*n_k))n_k^*\\
&=\sum_{j,k=1}^mF(Tn_j)F(n_j^*n_k)n_k^*=\sum_{j=1}^mF(Tn_j)n_j^*=T.
\end{align*}
\end{proof}

\begin{lem} 
For any $k$, we have 
$P_{2k}=P_{2}$ and $P_{2k+1}=P_{1}$. 
\end{lem}

\begin{proof}
We have only to show that $P_{2}\subseteq P_{2k}$ and $P_{1}\subseteq P_{2k+1}$. 
We observe that 
$$
P_2\ni z_{n_1}z_{n_2}=
\sum_{p_1,\cdots,p_{k-1}=1}^\infty
z_{n_1}z_{n_2}(z_{p_1}{z_{p_1}}^*)\cdots(z_{p_{k-1}}{z_{p_{k-1}}}^*)\in P_{2k}.
$$
Thus we have $P_{2}\subseteq P_{2k}$. Similarly we see that  $P_{1}\subseteq P_{2k+1}$. 

\end{proof}

\begin{prop}
$B(H)^{\phi^{2k}}=B(H)^{\phi^2}$ and $B(H)^{\phi^{2k+1}}=B(H)^{\phi}$. 
\end{prop}

\begin{proof}
By the previous lemma, we have $P_{2k}=P_{2}$. Then we have the inclusions 
$$
\begin{matrix}
P_{2k}&=& P_2&\subseteq &M&
\subseteq
&\langle 
M,e_2
\rangle&
=
&\langle 
M,e_{2k}
\rangle\\
 & & & & & &\cap& &\cap\\
 & & & & & & 
B(H)^{\phi^2}&\subseteq &B(H)^{\phi^{2k}}
\end{matrix}
$$ 
We have a normal conditional expectation 
$$
F(T)=\frac{1}{k}(T+\phi^2(T)+\cdots+\phi^{2(k-1)})
$$
from $B(H)^{\phi^{2k}}$ onto $B(H)^{\phi^{2}}$. 
Then by theorem 3.3, we have $F=id$ and hence 
$B(H)^{\phi^{2k}}=B(H)^{\phi^2}$. 
By the same way, we can see that $B(H)^{\phi^{2k+1}}=B(H)^{\phi}$. 
\end{proof}

\ \\ 

Next we consider two completely positive maps 
$\phi=\sum_{n=1}^\infty{\rm Ad}Jz_nJ\in G$ $(z_n\in M)$ 
and 
$\psi=\sum_{m=1}^\infty{\rm Ad}Jw_mJ\in G$ $(w_m\in M)$ 
such that  
$\sum_{n=1}^\infty z_nz_n^*=\sum_{n=1}^\infty z_n^*z_n=
\sum_{m=1}^\infty w_mw_m^*=\sum_{m=1}^\infty w_m^*w_m
=1$ 
and both $\{z_n\}_{n=1}^\infty$ and $\{w_m\}_{m=1}^\infty$ are 
$*$-closed sets. We also assume that $z_nw_m=w_mz_n$ for any $n,m$. 
Thus we have $\phi\circ\psi=\psi\circ\phi$. 
Let $P_1$ and $P_2$ be von Neumann algebras generated by 
$\{z_n\}_{n=1}^\infty$ and 
$\{z_{n_1}z_{n_2}\}_{n_1,n_2=1}^\infty$ respectively. 
Similarly, let $Q_1$ and $Q_2$ be von Neumann algebras generated by 
$\{w_m\}_{m=1}^\infty$ and 
$\{w_{m_1}w_{m_2}\}_{m_1,m_2=1}^\infty$ respectively. 
We further assume that ${P_2}'\cap P_1={Q_2}'\cap Q_1={\Bbb C}$, 
$P_2\not=P_1$ and $Q_2\not=Q_1$. 
Let $e_1$ and $e_2$ be the Jones projection onto $L^2(P_1\vee Q_1,\tau)$ 
and  $L^2(P_2\vee Q_2,\tau)$ respectively. 
We remark that 
$P_1\vee Q_1\simeq P_1\otimes Q_1$ and 
$P_2\vee Q_2\simeq P_2\otimes Q_2$. Then we have the following inclusions. 
$$
\begin{matrix}
P_2\vee Q_2&\subseteq& P_1\vee Q_1&\subseteq &M&
\subseteq
&\langle 
M,e_1
\rangle&
\subseteq
&\langle 
M,e_{2}
\rangle\\
 & & & & & &\cap& &\cap\\
 & & & & & & 
B(H)^{\phi}\cap B(H)^{\psi}&\subseteq &B(H)^{\phi^{2}}\cap B(H)^{\psi^2}.
\end{matrix}
$$ 
Here we remark that by proposition 3.4 and corollary 3.6, we have 
$B(H)^{\phi}\cap B(H)^{\psi}=B(H)^{\frac{1}{2}(\phi+\psi)}$ and 
$B(H)^{\phi^{2}}\cap B(H)^{\psi^2}=B(H)^{\phi^2\psi^2}$. 

\begin{lem}
The inclusions $B(H)^{\phi}\cap B(H)^{\psi}\subseteq 
B(H)^{\phi\psi}\subseteq 
B(H)^{\phi^2\psi^2}$ is isomorphic to 
inclusions of von Neumann algebras as operator systems. 
\end{lem}

\begin{proof}
Since $\phi\circ\psi=\psi\circ\phi$, $\phi$ and $\psi$ 
can be dilated 
simultaneously 
\cite{B98,S}. Then by the Prunaru theorem \cite{Pru12}, we are done. 
\end{proof}

\begin{lem}
The inclusion $B(H)^{\phi}\cap B(H)^{\psi}\subseteq 
B(H)^{\phi^2\psi^2}$ is irreducible.
\end{lem}

\begin{proof}
By theorem 2.3, we have $M'\cap B(H)^{\phi^2\psi^2}=M'\cap \langle 
M,e_2
\rangle$. Then we have 
$
(B(H)^{\phi}\cap B(H)^{\psi})'\cap B(H)^{\phi^2\psi^2}=(B(H)^{\phi}\cap B(H)^{\psi})'
\cap M'\cap B(H)^{\phi^2\psi^2}
=(B(H)^{\phi}\cap B(H)^{\psi})'\cap M'\cap \langle M,e_2\rangle\subseteq \langle 
M,e_1\rangle'\cap \langle M,e_2\rangle=J(P_2\vee Q_2)'\cap (P_1\vee Q_1)J={\Bbb C}
$. 
\end{proof}

\begin{lem}
We set $F(T)=\frac{1}{4}(T+\phi(T)+\psi(T)+\phi\circ\psi(T))$. Then $F$ is a normal conditional 
expectation from $B(H)^{\phi^2\psi^2}$ onto $B(H)^{\phi}\cap B(H)^{\psi}$. Moreover 
$F(\langle 
M,e_2
\rangle)\subseteq\langle 
M,e_1
\rangle$.  
\end{lem}

\begin{proof} Since $B(H)^{\phi^{2}}\cap B(H)^{\psi^2}=B(H)^{\phi^2\psi^2}$, 
the first statement is obvious. 
We will show that $F(\langle 
M,e_2
\rangle)\subseteq\langle 
M,e_1
\rangle$. 
We remark that 
$$
F(T)=\Bigl(
\dfrac{id+\phi}{2}
\Bigr)\circ\Bigl(
\dfrac{id+\psi}{2}
\Bigr)(T)=\Bigl(
\dfrac{id+\psi}{2}
\Bigr)\circ\Bigl(
\dfrac{id+\phi}{2}
\Bigr)(T)
$$
For any $T\in \langle 
M,e_2
\rangle=(\{Jz_{n_1}z_{n_2}J\}_{n_1,n_2=1}^\infty)'
\cap (\{Jw_{m_1}w_{m_2}J\}_{m_1,m_2=1}^\infty)'
$, 
since $\Bigl(
\dfrac{id+\psi}{2}
\Bigr)(T)\in (\{Jz_{n_1}z_{n_2}J\}_{n_1,n_2=1}^\infty)'$, 
by the proof of lemma 4.3, 
we have 
$
Jz_kJF(T)=F(T)Jz_kJ
$. Similarly we have $
Jw_kJF(T)=F(T)Jw_kJ
$ and hence $F(\langle 
M,e_2
\rangle)\subseteq\langle 
M,e_1
\rangle$.

\end{proof}

\begin{prop}
$[P_1\vee Q_1:P_2\vee Q_2]=[B(H)^{\phi^2\psi^2}:B(H)^{\phi}\cap B(H)^{\psi}]=4$.
\end{prop}

\begin{proof}
By proposition 4.4, we have $[P_1:P_2]=[Q_1:Q_2]=2$ and hence $[P_1\vee Q_1:P_2\vee Q_2]=4$. 
For any positive element $T\in B(H)^{\phi^2\psi^2}$, since 
$F(T)\geq \frac{1}{4}T$, 
by the Pimsner-Popa theorem \cite{PP,Po} 
we have $[B(H)^{\phi^2\psi^2}:B(H)^{\phi}\cap B(H)^{\psi}]\leq 4$. 
On the other hand, for any positive element $x\in \langle 
M,e_{2}
\rangle\subseteq  B(H)^{\phi^2\psi^2}$, we have 
$
F(x)\geq\dfrac{1}{[B(H)^{\phi^2\psi^2}:B(H)^{\phi}\cap B(H)^{\psi}]}x
$. Since $[\langle 
M,e_{2}
\rangle:\langle 
M,e_{2}
\rangle]=[P_1\vee Q_1:P_2\vee Q_2]=4$, we have 
$\dfrac{1}{4}\geq \dfrac{1}{[B(H)^{\phi^2\psi^2}:B(H)^{\phi}\cap B(H)^{\psi}]}$ 
and hence $[B(H)^{\phi^2\psi^2}:B(H)^{\phi}\cap B(H)^{\psi}]\geq 4$. 
Therefore we have $[B(H)^{\phi^2\psi^2}:B(H)^{\phi}\cap B(H)^{\psi}]=4$.
\end{proof}

\begin{cor} 
$B(H)^{\phi}\cap B(H)^{\psi}\not=
B(H)^{\phi\psi}$.
\end{cor}

\begin{proof}
Let $F_0(T)=\frac{1}{2}(T+\phi\psi(T))$. Then $F_0$ is a normal conditional 
expectation from $B(H)^{\phi^2\psi^2}$ onto $B(H)^{\phi\psi}$. Then for any 
positive $T\in B(H)^{\phi^2\psi^2}$, we have 
$
F_0(T)\geq\frac{1}{2}T
$ and hence $[B(H)^{\phi^2\psi^2}:B(H)^{\phi\psi}]\leq2$. Since 
 $B(H)^{\phi}\cap B(H)^{\psi}\subseteq 
B(H)^{\phi\psi}\subseteq 
B(H)^{\phi^2\psi^2}$ and $[B(H)^{\phi^2\psi^2}:B(H)^{\phi}\cap B(H)^{\psi}]=4$, 
we conclude $B(H)^{\phi}\cap B(H)^{\psi}\not=
B(H)^{\phi\psi}$. 
\end{proof}

\begin{ex} 
Let ${\Bbb F}_2$ be a free group with two generators $a$ and $b$. Let 
$A$ be a subgroup generated by $ab$ and $ab^{-1}$. 
Then 
$A$ consists of reduced words 
with even length. Set $M=L({\Bbb F}_2)$, $H=L^2({\Bbb F}_2)$ and 
$\phi=\frac{1}{4}({\rm Ad}JaJ+{\rm Ad}JbJ+{\rm Ad}Ja^{-1}J+{\rm Ad}Jb^{-1}J)
\in G$. In this case, $P_1=L({\Bbb F}_2)=M$ and $P_2=L(A)$. Since 
$[{\Bbb F}_2:A]=2$, we have $[P_1:P_2]=2$ and ${P_2}'\cap P_1={\Bbb C}$. 
We have the following inclusions.  
$$
\begin{matrix}
P_{2}&\subseteq& P_1&= &M&
\subseteq
&\langle 
M,e_{2}
\rangle\\
 & & & &\cap& &\cap\\
 & & & & 
B(H)^{\phi}&\subseteq &B(H)^{\phi^{2}}
\end{matrix}
$$ 
Then the subfactor $B(H)^{\phi}\subseteq B(H)^{\phi^{2}}$ is AFD type III 
with index $2$. The subfactors $M
\subseteq
\langle 
M,e_{2}
\rangle$ and  $B(H)^{\phi}\subseteq B(H)^{\phi^{2}}$ have a common 
Pimsner-Popa basis. Moreover two inclusions $M\subset B(H)^{\phi}$ and 
$\langle 
M,e_{2}
\rangle\subset B(H)^{\phi^{2}}$ have the rigidity property in the sence of 
theorem 3.3. 

Next we consider two unital normal completely positive 
maps $\phi\otimes id$ and $id\otimes\phi$ on $B(H\otimes H)$. 
Then by corollary 4.12, we have 
$
B(H\otimes H)^{\phi\otimes id}\cap B(H\otimes H)^{id\otimes\phi}
\not=B(H\otimes H)^{\phi\otimes \phi}
$. 
\end{ex}

\end{document}